\documentclass{amsart}

\usepackage{amsmath}
\usepackage{amssymb}
\usepackage{commath}
\usepackage{mathtools}
\usepackage{tikz-cd}

\usepackage{hyperref}

\allowdisplaybreaks

\theoremstyle{plain}
\newtheorem{theorem}{Theorem}
\newtheorem{lemma}[theorem]{Lemma}
\newtheorem{proposition}[theorem]{Proposition}
\newtheorem{corollary}[theorem]{Corollary}
\theoremstyle{definition}
\newtheorem{definition}[theorem]{Definition}
\newtheorem{remark}[theorem]{Remark}

\numberwithin{theorem}{section}
\numberwithin{equation}{section}

\newcommand{\C}{\mathbb{C}}

\newcommand{\R}{\mathbb{R}}

\newcommand{\re}{\operatorname{Re}}

\newcommand{\cP}{\mathcal{P}}
\newcommand{\cH}{\mathcal{H}}

\newcommand{\cB}{\mathcal{B}}

\newcommand{\cA}{\mathcal{A}}

\newcommand{\cF}{\mathcal{F}}

\newcommand{\mfT}{\mathfrak{T}}
\newcommand{\mfA}{\mathfrak{A}}
\newcommand{\mfC}{\mathfrak{C}}

\begin{document}

\title[Bargmann for translation invariant operators]{A Bargmann transform for translation invariant operators on weighted Bergman spaces of the complex half-plane}

\author{Ra\'ul Quiroga-Barranco}
\address{Centro de Investigaci\'on en Matem\'aticas, Guanajuato, M\'exico}
\email{quiroga@cimat.mx}

\keywords{Bergman space, Bargmann transform, Toeplitz operator}

\subjclass[2020]{Primary 46E22 47B32; Secondary 47B35 47C15}

\maketitle

\begin{abstract}
	Let us denote with $\mathcal{A}^2_\lambda(\mathbb{C}_+)$ ($\lambda > -1$) a weighted Bergman space over the right half-plane $\mathbb{C}_+$, which admits a unitary representation of $\mathbb{R}$ given by the (imaginary) translations of $\C_+$. We study the von Neumann algebra $\mathfrak{T}\big(\mathcal{A}^2_\lambda(\mathbb{C}_+)\big)$ of bounded translation invariant operators. We prove that every element of $\mathfrak{T}\big(\mathcal{A}^2_\lambda(\mathbb{C}_+)\big)$ is a Toeplit operator $T^{(\lambda)}_A$ for some translation invariant operator $A$ of the ambient $L^2$-space of $\mathcal{A}^2_\lambda(\mathbb{C}_+)$. Furthermore, we prove that this can be achieved through a commutative von Neumann algebra $\mathfrak{A}_\lambda$ that yields an assignment $A \mapsto T^{(\lambda)}_A$ that turns out to be a $*$-isomorphism of $*$-algebras.	Our main tool is a Bargmann transform $\mathcal{B}_\lambda$ for which we establish several operator and representation theoretic properties. We also describe the translation invariant subspaces of $\mathcal{A}^2_\lambda(\mathbb{C}_+)$ and obtain formulas for the diagonalizing spectral functions for translation invariant Toeplitz operators whose symbols are translation invariant operators. The latter generalize previously known results for function symbols.
\end{abstract}


\section{Introduction} 
\label{sec:Intro}
Bergman spaces and Toeplitz operators is an interesting topic of research in functional analysis and operator theory. One approach to such topic is based on the use of symmetries for the domain under consideration and how they are reflected on the behavior of operators. For the right half-plane, denoted here with $\C_+$, the one-parameter group of translations yields unitary representations on the corresponding weighted Bergman spaces $\cA^2_\lambda(\C_+)$ ($\lambda > -1$), which are then used to study Toeplitz operators by considering those that intertwine (i.e.~ commute with) such representations. Some relevant works along this idea are \cite{GKVParabolic,GQVUnitDisk,HHMVerticalWeighted,HMVVertical}. Note that these references make use of the upper half-plane instead of the right half-plane $\C_+$; it is needless to say that our choice does not change the theory in any essential way. 

About the previous research we may single out two important features. First, a so-called Bargmann transform is used to unitarily map Bergman spaces onto $L^2$-spaces, but only as a secondary tool to achieve the desired properties and results. Second, Toeplitz operators are studied only for the case of function symbols, mostly ignoring the case of operator symbols.

The contribution of this work is two-fold, which intends to improve the previous features just described. In the first place, we present a Bargmann transform $\cB_\lambda$ from $L^2(\R_+)$ onto $\cA_\lambda(\C_+)$ as an operator whose study is interesting in its own right. In second place, we study Toeplitz operators whose symbols are bounded operators themselves and not just functions. This is all developed in the setup of the unitary representations on $\cA^2_\lambda(\C_+)$ given by the translations on $\C_+$. As we will explain, this allows us to obtain considerably stronger results than those from previous works.

Our Bargmann transform $\cB_\lambda$ is specifically constructed for the unitary representations defined by translations. It is build out of the isometric embedding $V_\lambda$ of $L^2(\R_+)$ into $L^2(\C_+, v_\lambda)$ defined in Proposition~\ref{prop:Vlambda}, where $L^2(\C_+, v_\lambda)$ is the ambient $L^2$-space of $\cA^2_\lambda(\C_+)$ (see~section~\ref{sec:BergmanToeplitz}). It also makes use of a unitary map $\cF_\lambda$ of $L^2(\C_+, v_\lambda)$ defined by applying Fourier transform in the imaginary part variable (see~section~\ref{sec:Bargmann}). The first main result, Theorem~\ref{thm:Bargmann-diagram}, highlights the relevance of $\cB_\lambda$: the pairs of Hilbert spaces $(L^2(\R_+), L^2(\C_+, v_\lambda))$ and $(\cA^2_\lambda(\C_+), L^2(\C_+, v_\lambda))$ have a natural correspondence through the pair of unitary operators $(\cB_\lambda, \cF_\lambda)$ (see~diagram~\eqref{eq:Bargmann-diagram}). In Corollary~\eqref{cor:Bargmann-formulas} we obtain formulas for $\cB_\lambda$ and its inverse that clearly resemble those for the classical Fock-Segal-Bargmann transform.

In section~\ref{sec:BargmannIntertwining} we prove that the Bargmann transform $\cB_\lambda$ and the unitary operator $\cF_\lambda$ intertwine the unitary representations given by multiplication by characters in the pair $(L^2(\R_+), L^2(\C_+, v_\lambda))$ and by translations in the pair $(\cA^2_\lambda(\C_+), L^2(\C_+, v_\lambda))$ (see~Corollary~\ref{cor:Bargmann-intertwining}). This is a very important property, ignored until now, whose explicit application should be able to advance our knowledge of Bergman spaces. In this work, we use it in Theorem~\ref{thm:cA-invariant-subspaces} to describe the translation invariant subspaces of $\cA^2_\lambda(\C_+)$; these turn out be in one-to-one correspondence with the Borel subsets of $\R_+$ through $\cB_\lambda$ (see~also~Corollary~\ref{cor:cA-invariant-subspaces}).

In section~\ref{sec:TranslationInvariantOperators} we start our study of translation invariant operators of $\cA^2_\lambda(\C_+)$, either Toeplitz or not, and with operator symbols in the former case. It is already known (see~\cite{HHMVerticalWeighted,HMVVertical}) that such operators admit a diagonalizing function $h \in L^\infty(\R_+)$ with respect to $\cB_\lambda$. We provide a self-contained elementary proof of this fact in Proposition~\ref{prop:TranslationInvariance}. We call such diagonalizing functions Bargmann spectral (see~Definition~\ref{def:spectralFromBargmann}) to emphasize that they originate from $\cB_\lambda$.

The next main result is Theorem~\ref{thm:TranslationInvariantOperatorsVonNeumann}, where we give a complete description of all translation invariant operators of $\cA^2_\lambda(\C_+)$, which we denote by $\mfT\big(\cA^2_\lambda(\C_+)\big)$. We prove that any element of the latter is a Toeplitz operator but with an operator symbol. Furthermore, we construct a von Neumann algebra $\mfA_\lambda$ of operators acting on $L^2(\C_+, v_\lambda)$ such that the assignment $A \mapsto T^{(\lambda)}_A$ from $\mfA_\lambda$ onto $\mfT\big(\cA^2_\lambda(\C_+)\big)$ is a $*$-isomorphism of $*$-algebras (not just vector spaces). The non-triviality of these and other properties are explained in Remark~\ref{rmk:TranslationInvariantOperatorsVonNeumann}.

On the other hand, Theorem~\ref{thm:AlambdaConvolutionType} decodes the structure of the von Neumann algebra $\cA_\lambda$. It is the algebra of operators of the form $I \otimes C$, where $C$ is a convolution type operator associated to a function belonging to $L^\infty(\R_+)$, with respect to the tensor product decomposition $L^2(\C_+, v_\lambda) = L^2(\R_+, \widehat{v}_\lambda) \otimes L^2(\R)$ (see~\ref{eq:L2-tensordecomposition} and the related discussion).

In section~\ref{sec:MoreTranslationInvOp} we consider further operator symbols for our Toeplitz operators: those that intertwine $\cF_\lambda$ and operators that we call imaginary part independent. The latter are those of the form $\widehat{A} \otimes I$ with respect to the tensor product decomposition \eqref{eq:L2-tensordecomposition} (see~Definition~\ref{def:ImPartInd}). It turns out that a multiplier operator $M_a$, where $a \in L^\infty(\C_+)$, is an imaginary part independent operator if and only if $a$ does not depend on its imaginary part variable (see~Corollary~\ref{cor:ImPartIndSymbols}). We also refer to Proposition~\ref{prop:ImPartInd-withfunction} for a further interpretation of our notation.

Finally, in section~\ref{sec:Bargmann-spectral} we obtain formulas for the Bargmann spectral functions of Toeplitz operators with suitable translation invariant operator symbols. Theorem~\ref{thm:spectral-FourierIntertwining} yields a formula for the Bargmann spectral function in the case of operators intertwining $\cF_\lambda$; in this case, the formula only depends on the operator symbol $A$, the isometric embedding $V_\lambda$ and its adjoint $V_\lambda^*$. Theorem~\ref{thm:ImPartInd-spectral} provides a formula for the Bargmann spectral function for imaginary part independent operator symbols that very closely resembles the formula known for function symbols. As a matter of fact, it is proved in Corollary~\ref{cor:spectral-FourierIntertwining} that our formulas for operator symbols obtained section~\ref{sec:Bargmann-spectral} have as a particular case those for function symbols.

\section{Bergman spaces and Toeplitz operators on $\C_+$}
\label{sec:BergmanToeplitz}
In this section we recall some well-known facts about Bergman spaces on a complex half-plane and the Toeplitz operators on such spaces. For most claims below, our main references are \cite{GKVParabolic,GQVUnitDisk,HHMVerticalWeighted,HMVVertical,QVUnitBall1,UpmeierToepBook}. Note that we will make use of the right half-plane, while most of these references introduce Bergman spaces and Toeplitz operators using the upper half-plane. However, \cite{UpmeierToepBook} considers Bergman spaces corresponding to the left half-plane. Nevertheless, rotations in $\C$ around the origin easily show that our claims in this section follow directly from the results in these references.

Let us denote. from now on. the complex right half-plane by
\[
	\C_+ = \{ z \in \C : \re(z) > 0\}.
\]
For every weight $\lambda > -1$, we define
\[
	\dif v_\lambda(z) = \frac{\lambda + 1}{4 \pi} \big(\re(z)\big)^\lambda \dif z,
\]
as a measure in $\C_+$. The weighted Bergman space corresponding to the weight $\lambda > -1$ is thus defined by
\[
	\cA^2_\lambda(\C_+) = \{ f \in L^2(\C_+, v_\lambda) : f \text{ is holomorphic} \}.
\]
Every weighted Bergman space is well known to be a reproducing kernel Hilbert space. The Bergman kernel of $\cA^2_\lambda(\C_+)$ is given by
\begin{align*}
	K_\lambda : \C_+ \times \C_+ &\rightarrow \C \\
	K_\lambda(z, \zeta) &= \bigg(\frac{z + \overline{\zeta}}{2}\bigg)^{-(\lambda + 2)}.
\end{align*}
In particular, the Bergman projection $P_\lambda : L^2(\C_+,v_\lambda) \rightarrow \cA^2_\lambda(C_+)$ can be written~as
\begin{align*}
	\big(P_\lambda(f)\big)(z) &= \int_{\C_+} f(\zeta)
				\bigg(\frac{z + \overline{\zeta}}{2}\bigg)^{-(\lambda + 2)}
					\dif v_\lambda(\zeta) \\
			&= \frac{\lambda + 1}{4\pi}\int_{\C_+} f(\zeta) 
				\bigg(\frac{z + \overline{\zeta}}{2}\bigg)^{-(\lambda + 2)}
					\big(\re(\zeta)\big)^\lambda \dif\zeta  \\
			&= \frac{\lambda + 1}{\pi} \int_{\C_+} f(\zeta)
				\frac{\big(2\re(\zeta)\big)^\lambda}{(z + \overline{\zeta})^{\lambda + 2}} 
							\dif \zeta,
\end{align*}
for every $f \in \cA^2_\lambda(\C_+)$ and $z \in \C_+$.

In this work, we will consider Toeplitz operators obtained from arbitrary bounded operators of the ambient $L^2$-space. For every $\lambda > -1$ and a bounded operator $A :  L^2(\C_+,v_\lambda) \rightarrow L^2(\C_+, v_\lambda)$, the \textbf{Toeplitz operator with symbol $A$} is defined~by
\begin{align*}
	T_A = T^{(\lambda)}_A : \cA^2_\lambda(\C_+) &\rightarrow \cA^2_\lambda(\C_+)  \\
		T_A(f) &= P_\lambda(A(f)),
\end{align*}
where the super-index $(\lambda)$ is sometimes omitted when it is clear from the context. Note that an equivalent expression for the Toeplitz operator above is given by
\[
	T^{(\lambda)}_A = P_\lambda \circ A \circ P_\lambda^*. 
\]
The reason is that $P_\lambda^* : \cA^2_\lambda(\C_+) \hookrightarrow L^2(\C_+, v_\lambda)$ is the natural inclusion.

As a particular case, for every $a \in L^\infty(\C_+)$ and the corresponding multiplier operator $M_a$ of $L^2(\C_+,v_\lambda)$ we use the notation
\[
	T^{(\lambda)}_a = T^{(\lambda)}_{M_a}.
\]
However, we emphasize our use of arbitrary bounded operators of $L^2(\C_+, v_\lambda)$ as symbols for Toeplitz operators. This is in contrast with constructions found in our references and most of the literature on this topic, where only measurable functions, through their multiplier operators, are used as symbols of Toeplitz operators.

\section{A Bargmann transform for $\cA^2_\lambda(\C_+)$}
\label{sec:Bargmann}
Note that the definition of the measure $v_\lambda$ clearly yields the tensor product decomposition
\begin{equation}\label{eq:L2-tensordecomposition}
	L^2(\C_+, v_\lambda) = L^2(\R_+, \widehat{v}_\lambda) 
		\otimes L^2(\R),
\end{equation}
where we define the measure $\widehat{v}_\lambda$ on $\R_+$ by
\begin{equation}\label{eq:widehatv_lambda}
	\dif \widehat{v}_\lambda(x) = \frac{\lambda + 1}{4\pi} x^\lambda \dif x.
\end{equation}
We will assume this tensor product decomposition in the rest of this work. Moreover, we will make use of the notation and well-known properties for tensor products of Hilbert and $L^2$-spaces, as well as their operators, without further mention. We refer to \cite[Appendix~A.3]{FollandHarmonic} for more details on tensor products.

Let us consider the Fourier transform in $L^2(\R)$ given by
\[
\big(\cF(g)\big)(\xi) = \frac{1}{\sqrt{2\pi}}
\int_\R g(y) e^{-i\xi y} \dif y,
\]
for $g \in L^1(\R) \cap L^2(\R)$. 
We introduce the unitary operator $\cF_\lambda = I \otimes \cF$ of $L^2(\C_+,v_\lambda)$ whose definition is obtained from the tensor product decomposition \eqref{eq:L2-tensordecomposition}. The next result, characterizes the elements of $\cA^2_\lambda(\C_+)$ up to the operator $\cF_\lambda$.

\begin{proposition}\label{prop:CR-Fourier}
	With the previous notation, a function $\psi \in L^2(\C_+, v_\lambda)$ satisfies $\cF_\lambda(\psi) \in \cA^2_\lambda(\C_+)$ if and only if there exists a measurable function $\varphi : \R_+ \rightarrow \C$ such~that
	\[
		\psi(x + i\xi) = e^{-x\xi} \varphi(\xi) \chi_{\R_+}(\xi),
	\]
	for almost every $x + i \xi \in \C_+$, where $\varphi$ satisfies the integrability condition
	\[
		\int_{\R_+} \frac{|\varphi(\xi)|^2}{\xi^{\lambda + 1}} \dif \xi < +\infty.
	\]
\end{proposition}
\begin{proof}
	By the properties of Fourier transform, we have
	\[
		\cF_\lambda^{-1} \circ \frac{\partial}{\partial y} \circ \cF_\lambda
				= M_{-i\xi}
	\]
	on a dense subset of $L^2(\C_+, v_\lambda)$. Hence, Cauchy-Riemann equations in the variable $z = x + iy$ correspond to the equation
	\[
		\Big(\frac{\partial}{\partial x} + \xi\Big) \psi = 0
	\]
	for the variable $w = x + i\xi$, when $\psi$ belongs to a suitable dense subset of $L^2(\C_+, v_\lambda)$. The general solution of this equation is given by functions of the form
	\[
		\psi(x + i\xi) = e^{-x\xi} \varphi(\xi),
	\]
	for some $\varphi : \R \rightarrow \C$. Hence, a function $\psi$ of this sort satisfies $\cF_\lambda(\psi) \in \cA^2_\lambda(\C_+)$ if and only it also belongs to $L^2(\C_+, v_\lambda)$. We now compute the $L^2$-norm of such a function $\psi$ as follows
	\begin{align*}
		\|\psi\|^2 &= \frac{\lambda + 1}{4\pi} \int_{\R_+ \times \R}
			e^{-2x\xi} |\varphi(\xi)|^2 
						x^\lambda \dif x \dif \xi   \\
				&= \frac{\lambda + 1}{4\pi} \int_\R |\varphi(\xi)|^2
					\Big(\int_{\R_+} e^{-2x\xi} x^\lambda \dif x \Big) \dif \xi	 \\
		\intertext{for this to be finite, it is necessary and sufficient to have $\xi \in \R_+$ in the integral between parenthesis, and in this case}
				&= \frac{\lambda + 1}{4\pi} \int_{\R_+} |\varphi(\xi)|^2
					\frac{\Gamma(\lambda + 1)}{(2\xi)^{\lambda + 1}} \dif \xi  \\
				&= \frac{\Gamma(\lambda + 2)}{\pi2^{\lambda + 3}}
					\int_{\R_+} \frac{|\varphi(\xi)|^2}{\xi^{\lambda + 1}}.
	\end{align*}
	This computation proves that for $\psi \in L^2(\C_+, v_\lambda)$ we have $\cF_\lambda(\psi) \in \cA^2_\lambda(\C_+)$ if and only if the expression in the statement holds, including the integrability condition on $\varphi$, thus completing the proof.
\end{proof}

The previous result and its proof leads us to the construction of a natural isometry of $L^2(\R_+)$ into $L^2(\C_+,v_\lambda)$.

\begin{proposition}\label{prop:Vlambda}
	The operator $V_\lambda : L^2(\R_+) \rightarrow L^2(\C_+, v_\lambda)$ given by
	\[
		\big(V_\lambda(\varphi)\big)(x + i\xi) =
			\sqrt{\frac{\pi 2^{\lambda + 3}\xi^{\lambda + 1}}{\Gamma(\lambda + 2)}}
					e^{-x\xi} \varphi(\xi) \chi_{\R_+}(\xi),
	\]
	for $\varphi \in L^2(\R_+)$ and for almost every $x + i\xi \in \C_+$, is an isometry that satisfies
	\[
		\big(\cF_\lambda \circ V_\lambda\big)\big(L^2(\R_+)\big) 
			= \cA^2_\lambda(\C_+).
	\]
	The adjoint operator $V_\lambda^*$ is given by
	\begin{align*}
		\big(V_\lambda^*(\psi)\big)(\xi) 
			&= \sqrt{\frac{2^\lambda(\lambda + 1)\xi^{\lambda + 1}}{2\pi \Gamma(\lambda + 1)}} 
				\int_{\R_+} \psi(x + i \xi) e^{-x\xi} x^\lambda \dif x  \\
			&= \sqrt{\frac{\pi 2^{\lambda + 3}\xi^{\lambda + 1}}{\Gamma(\lambda + 2)}}
				\int_{\R_+} \psi(x + i \xi) e^{-x\xi} \dif \widehat{v}_\lambda(x)
	\end{align*}
	for every $\psi \in L^2(\C_+, v_\lambda)$ and for almost every $\xi \in \R_+$.
\end{proposition}
\begin{proof}
	The operator $V_\lambda$ is clearly linear.	For a given function $\varphi \in L^2(\R_+)$, we compute the norm of $V_\lambda(\varphi)$ in $L^2(\C_+, v_\lambda)$ as follows.
	\begin{align*}
		\|V_\lambda(\varphi)\|^2 &=
			\frac{\lambda + 1}{4\pi}\frac{\pi 2^{\lambda + 3}}{\Gamma(\lambda + 2)}
				\int_{\R_+^2} \xi^{\lambda + 1} e^{-2x\xi}|\varphi(\xi)|^2 x^\lambda 
						\dif x \dif \xi \\
			&= \frac{2^{\lambda + 1}}{\Gamma(\lambda + 1)}
				\int_{\R_+} \xi^{\lambda + 1} |\varphi(\xi)|^2
					\Big(
						\int_{\R_+} e^{-2x\xi} x^\lambda \dif x
					\Big) \dif \xi  \\
			&= \|\varphi\|^2,
	\end{align*}
	where we have used the previously computed Gamma integral involved. This proves that $V_\lambda$ is indeed an isometry.
	
	We now obtain the expression for $V_\lambda^*$ through the following computation. For every $\varphi \in L^2(\R_+)$ and $\psi \in L^2(\C_+,v_\lambda)$ we have
	\begin{align*}
		\langle V_\lambda(\varphi),\psi \rangle &=
			\frac{\lambda + 1}{4\pi} \sqrt{\frac{\pi 2^{\lambda + 3}}{\Gamma(\lambda + 2)}}
				\int_{\R_+^2} \xi^{\frac{1}{2}(\lambda + 1)} e^{-x\xi} \varphi(\xi)
					\overline{\psi(x + i\xi)} x^\lambda \dif x \dif \xi  \\
			&= \sqrt{\frac{2^\lambda (\lambda + 1)}{2\pi \Gamma(\lambda +1)}}
				\int_{\R_+}\varphi(\xi) \xi^{\frac{1}{2}(\lambda + 1)}
					\Big(
						\int_{\R_+} \overline{\psi(x + i \xi)} e^{-x\xi} x^\lambda \dif x
					\Big) \dif \xi,
	\end{align*}
	which yields the first stated expression for $V_\lambda^*$. The second stated expression for $V_\lambda^*$ is obtained from the definition \eqref{eq:widehatv_lambda} of $\widehat{v}_\lambda$.
\end{proof}

\begin{remark}\label{rmk:Vlambda-inclusion}
	Since $V_\lambda : L^2(\R_+) \rightarrow L^2(\C_+, v_\lambda)$ is an isometry, we will consider, from now on, $L^2(\R_+)$ as a Hilbert subspace of $L^2(\C_+, v_\lambda)$. In other words, we will identify $L^2(\R_+)$ with the Hilbert subspace $V_\lambda\big(L^2(\R_+)\big)$ of $L^2(\C_+, v_\lambda)$. We note that, with this convention, the adjoint operator $V_\lambda^* : L^2(\C_+, v_\lambda) \rightarrow L^2(\R_+)$ is thus precisely the corresponding orthogonal projection.
\end{remark}

\begin{theorem}\label{thm:Bargmann-diagram}
	Let us define the operator $\cB_\lambda = \cF_\lambda \circ V_\lambda : L^2(\R_+) \rightarrow \cA^2_\lambda(\C_+)$. Then, $\cB_\lambda$ is well defined, i.e.~$\cB_\lambda\big(L^2(\R_+)\big) = \cA^2_\lambda(C_+)$. Moreover, the following diagram is commutative
	\begin{equation}\label{eq:Bargmann-diagram}
	\begin{tikzcd}
		L^2(\C_+, v_\lambda) \arrow[rr,"\cF_\lambda"] \arrow[d, bend right=60, "V_\lambda^*"']
			&& L^2(\C_+, v_\lambda) \arrow[d, bend left=60, "P_\lambda"] \\
		L^2(\R_+) \arrow[rr,"\cB_\lambda"] \arrow[u, hook', "V_\lambda"'] 
			&& \cA^2(\C_+) \arrow[u, hook]
	\end{tikzcd}
	\end{equation}
	and it satisfies the following properties.
	\begin{enumerate}
		\item The horizontal arrows are unitary.
		\item The upward vertical arrows are natural isometric inclusions.
		\item The downward vertical arrows are orthogonal projections.
	\end{enumerate}
\end{theorem}
\begin{proof}
	First, note that Propositions~\ref{prop:CR-Fourier} and \ref{prop:Vlambda} yield $\big(\cF_\lambda \circ V_\lambda\big)\big(L^2(\R_+)\big) = \cA^2_\lambda(\C_+)$, which implies the first claim. Next, note that the vertical arrows on the right are, by definition, the natural inclusion (upward arrow) and the Bergman orthogonal projection $P_\lambda$.
	
	By Proposition~\ref{prop:Vlambda}, and as noted in Remark~\ref{rmk:Vlambda-inclusion}, the operator $V_\lambda$ is isometric and thus considered as the inclusion $L^2(\R_+) \hookrightarrow L^2(\C_+, v_\lambda)$. With this convention, the operator $V_\lambda^*$ is thus the orthogonal projection from $L^2(\C_+, v_\lambda)$ onto $L^2(\R_+)$. Hence, properties (2) and (3) have now been established.
	
	By construction, the operator $\cF_\lambda$ is unitary. On the other, we already have established the identity $\big(\cF_\lambda \circ V_\lambda\big)\big(L^2(\R_+)\big) = \cA^2_\lambda(\C_+)$. Hence, the operator $\cB_\lambda$ is unitary from $L^2(\R_+)$ onto $\cA^2_\lambda(\C_+)$ because $V_\lambda$ is an isometry. This proves property~(1).
	
	For the diagram above, the previous paragraphs yield the commutativity for the horizontal and upward vertical arrows. Since the downward arrows are orthogonal projections, the full diagram is commutative.
\end{proof}

\begin{definition}\label{def:Bargmann}
	The unitary operator $\cB_\lambda = \cF_\lambda \circ V_\lambda : L^2(\R_+) \rightarrow \cA^2_\lambda(\C_+)$ will be called the \textbf{Bargmann transform for translation invariant operators of the weighted Bergman space $\cA^2_\lambda(\C_+)$}, or simply (in the rest of this work) the \textbf{Bargmann transform}.
\end{definition}

The next result computes the explicit integral expression for the Bargmann transform $\cB_\lambda$ and its inverse.

\begin{corollary}\label{cor:Bargmann-formulas}
	The Bargmann transform $\cB_\lambda : L^2(\R_+) \rightarrow \cA^2_\lambda(\C_+)$ satisfies
	\[
		\big(\cB_\lambda(\varphi)\big)(z) =
			\sqrt{\frac{2^{\lambda + 2}}{\Gamma(\lambda + 2)}}
			\int_{\R_+} \varphi(\xi) \xi^{\frac{1}{2}(\lambda + 1)} e^{-\xi z} \dif \xi,
	\]
	for every $\varphi \in L^2(\R_+)$ belonging to a dense subspace, and for every $z \in \C_+$. On the other hand, the inverse $\cB_\lambda^{-1} : \cA^2_\lambda(\C_+) \rightarrow L^2(\R_+)$ satisfies
	\[
		\big(\cB_\lambda^{-1}(f)\big)(\xi) =
			\sqrt{\frac{2^{\lambda + 2}}{\Gamma(\lambda + 2)}}
			\xi^{\frac{1}{2}(\lambda + 1)} 
			\int_{\C_+} f(z) e^{-\xi \overline{z}} \dif v_\lambda(z),
	\]
	for every $f \in \cA^2_\lambda(\C_+)$ belonging to a dense subspace, and for every $\xi \in \R_+$.
\end{corollary}
\begin{proof}
	For every $\varphi \in L^2(\R_+)$, in the dense subspace where the Fourier integral below converges, and for $z = x + iy \in \C_+$, we have
	\begin{align*}
		\big(\cB_\lambda(\varphi)\big)(z) 
			&= \big(\cF_\lambda \circ V_\lambda (\varphi) \big)(x + i y)  \\
			&= \frac{1}{\sqrt{2\pi}} \int_\R \big(V_\lambda(\varphi)\big)(x + i\xi) 
					e^{-iy\xi} \dif \xi  \\
			&= \frac{1}{\sqrt{2\pi}} \sqrt{\frac{\pi 2^{\lambda + 3}}{\Gamma(\lambda + 2)}}
				\int_{\R_+} \xi^{\frac{1}{2}(\lambda + 1)} e^{-x\xi} \varphi(\xi) 
					e^{-iy\xi} \dif \xi  \\
			&= \sqrt{\frac{2^{\lambda + 2}}{\Gamma(\lambda + 2)}}
				\int_{\R_+} \varphi(\xi) \xi^{\frac{1}{2}(\lambda + 1)} e^{-\xi z} \dif \xi.
	\end{align*}
	On the other hand, we have for $f \in \cA^2_\lambda(\C_+)$, in the dense subspace where the inverse Fourier integral converges, and $\xi \in \R_+$, the following computation
	\begin{align*}
		\big(\cB_\lambda^{-1}(f)\big)(\xi) 
			&= \big(V_\lambda^* \circ \cF_\lambda^{-1}(f)\big)(\xi)  \\
			&= \sqrt{\frac{\pi 2^{\lambda + 3}}{\Gamma(\lambda + 2)}}
					\xi^{\frac{1}{2}(\lambda + 1)}
					\int_{\R_+} \big(\cF_\lambda^{-1}(f)\big)(x + i\xi) 
						e^{-x\xi} \dif \widehat{v}(x)  \\
			&= \sqrt{\frac{\pi 2^{\lambda + 3}}{\Gamma(\lambda + 2)}}
					\xi^{\frac{1}{2}(\lambda + 1)}
					\frac{1}{\sqrt{2\pi}}
					\int_{\R_+ \times \R} f(x + iy) e^{-x\xi}
						e^{iy\xi} \dif \widehat{v}(x) \dif y \\
			&= \sqrt{\frac{2^{\lambda + 2}}{\Gamma(\lambda + 2)}}
					\xi^{\frac{1}{2}(\lambda + 1)}
					\int_{\R_+ \times \R} f(x + iy) e^{-\xi(x - iy)} 
							\dif \widehat{v}_\lambda(x) \dif y \\
			&= \sqrt{\frac{2^{\lambda + 2}}{\Gamma(\lambda + 2)}}
					\xi^{\frac{1}{2}(\lambda + 1)} 
					\int_{\C_+} f(z) e^{-\xi \overline{z}} \dif v_\lambda(z).
	\end{align*}
\end{proof}

\begin{remark}\label{rmk:Bargmann-formulas}
	We note that Corollary~\ref{cor:Bargmann-formulas} essentially recovers the formulas of  so-called Bargmann type transforms from \cite{GKVParabolic}. More precisely, Theorem~2.3 and Corollary~2.4 from \cite{GKVParabolic} construct unitary maps that correspond to our Bargmann transform $\cB_\lambda$ and its inverse $\cB_\lambda^{-1}$, respectively. Our formulas and those found in \cite{GKVParabolic} are basically the same except for two facts: a different normalization of the measure $v_\lambda$ and our use of the right half-plane instead of the upper half-plane as found in \cite{GKVParabolic}.
	
	On the other hand, our Bargmann transform $\cB_\lambda$ is obtained through a single unitary operator: a Fourier transform on the imaginary part. While, the Bargmann type transform considered in \cite{GKVParabolic} requires the introduction of an additional unitary transformation and some auxiliary functions (see~section~2 of the aforementioned reference). Our construction not only leads to simpler proof and presentation, but it also allows to work in the same ambient Hilbert space $L^2(\C_+, v_\lambda)$, as one easily verifies from the commutative diagram in Theorem~\ref{thm:Bargmann-diagram}. In comparison, \cite{GKVParabolic} replaces the ambient Hilbert space with $L^2(\R \times \R_+, \dif x \dif y)$, thus loosing some of the symmetry observed in Theorem~\ref{thm:Bargmann-diagram}, as well as in Corollary~\ref{cor:Bargmann-formulas}. 
	
	As we will see in the sections that follow, our approach allows to more easily make use of representation theoretic methods (not considered in \cite{GKVParabolic}) to obtain new applications to the more general Toeplitz operators introduced in section~\ref{sec:BergmanToeplitz}. Recall that, for the latter, a symbol is in general a bounded operator and thus not necessarily a multiplier operator.
\end{remark}

\section{Intertwining property of the Bargmann transform}
\label{sec:BargmannIntertwining}
We will consider two fundamental unitary representation of $\R$ on $L^2(\C_+, v_\lambda)$. These will be defined and denoted as follows in the rest of this work.
\begin{align*}
	\tau : \R \times L^2(\C_+, v_\lambda) &\rightarrow L^2(\C_+, v_\lambda) &
	\rho : \R \times L^2(\C_+, v_\lambda) &\rightarrow L^2(\C_+, v_\lambda) \\	
	\tau(t,f)(z) = \big(\tau_t(f)\big)(z) &= f(z - it) &
	\rho(t,\psi)(x + i\xi) = \big(\rho_t(\psi)\big)(z) &= e^{it\xi} \psi(x + i\xi).
\end{align*}
Note that the measure $v_\lambda$ is invariant under the translation by $it$ when $t \in \R$, which is the reason for $\tau$ to be a unitary representation. Also, multiplication by the character $\xi \mapsto e^{it\xi}$ clearly leaves invariant the inner product of $L^2(\C_+,v_\lambda)$, and so $\rho$ is unitary as well. The required continuity of the maps $\tau$ and $\rho$ are easily proved by standard methods.

Recall that $L^2(\R_+)$ is considered as a Hilbert subspace of $L^2(\C_+, v_\lambda)$ through the isometry $V_\lambda$.

\begin{lemma}\label{lem:cA-L2R-invariance}
	The Hilbert subspaces $L^2(\R_+)$ and $\cA^2_\lambda(\C_+)$ of $L^2(\C_+, v_\lambda)$ are invariant under the unitary representations $\rho$ and $\tau$, respectively.
\end{lemma}
\begin{proof}
	The $\rho$-invariance of $L^2(\R_+)$ follows easily from the expression of $V_\lambda$ (see~Proposition~\ref{prop:Vlambda}). It is also clear that $\cA^2_\lambda(\C_+)$ is invariant under composition with translation by $it$, for $t \in \R$, and so it is invariant under $\tau$.
\end{proof}

\begin{remark}\label{rmk:L2R-invariance}
	We observe that the unitary representation of $\R$ induced on $L^2(\R_+)$ from the $\rho$-invariance stated in Lemma~\ref{lem:cA-L2R-invariance} is the well-known representation obtained by character multiplication
	\begin{align*}
		\R \times L^2(\R_+) &\rightarrow L^2(\R_+)  \\
		\big(t\cdot g\big)(\xi) &= e^{it\xi} g(\xi).
	\end{align*}
	This is seen easily from the expression of $V_\lambda$. This unitary representation of $\R$ on $L^2(\R_+)$ will be denoted with the same symbol $\rho$. Furthermore, the latter is the restriction of a corresponding unitary representation of $\R$ on $L^2(\R)$ obtained with the same formula. Hence, we will also denote this representation of $\R$ on $L^2(\R)$ with the symbol $\rho$.
\end{remark}

To abbreviate the notation corresponding to invariance with respect to $\tau$ and $\rho$ as considered in Lemma~\ref{lem:cA-L2R-invariance} and Remark~\ref{rmk:L2R-invariance}, we introduce the following.

\begin{definition}\label{def:invariance}
	A Hilbert subspace $\cH$ of either $\cA^2_\lambda(\C_+)$ or $L^2(\C_+, v_\lambda)$ will be called \textbf{translation invariant} when $\tau_t(\cH) = \cH$, for every $t \in \R$, i.e.~when it is invariant under the unitary representation $\tau$. Correspondingly, a Hilbert subspace $\cH$ of either $L^2(\R_+)$ or $L^2(\C_+, v_\lambda)$ will be called \textbf{character invariant} when $\rho_t(\cH) = \cH$, for every $t \in \R$, i.e.~when it is invariant under the unitary representation $\rho$.
\end{definition}

It is very well known that the Fourier transform intertwines translations and character multiplication. This leads to the next immediate consequence.

\begin{lemma}\label{lem:Flambda-intertwining}
	With our current notation, we have
	\[
		\tau_t \circ \cF_\lambda = \cF_\lambda \circ \rho_t,
	\]
	for every $t \in \R$, as operators of $L^2(\C_+, v_\lambda)$.
\end{lemma}

Lemma~\ref{lem:Flambda-intertwining} induces a corresponding intertwining property for the Bargmann transform $\cB_\lambda$, whose proof illustrates the formulas obtained in Corollary~\ref{cor:Bargmann-formulas}.

\begin{corollary}\label{cor:Bargmann-intertwining}
	The Bargmann transform $\cB_\lambda : L^2(\R_+) \rightarrow \cA^2_\lambda(\C_+)$ intertwines the unitary representations $\rho$ and $\tau$ on $L^2(\R_+)$ and $\cA^2(\C_+, v_\lambda)$, respectively. More precisely, we have
	\[
		\tau_t \circ \cB_\lambda = \cB_\lambda \circ \rho_t,
	\]
	for every $t \in \R$.
\end{corollary}
\begin{proof}
	From Corollary~\ref{cor:Bargmann-formulas}, for every $\varphi$ in some dense subspace of $L^2(\R_+)$ and $t \in \R$, we have
	\begin{align*}
		\big(\tau_t \circ \cB_\lambda(\varphi)\big)(z) 
			&= \big(\cB_\lambda(\varphi)\big)(z - it) \\
			&= \sqrt{\frac{2^{\lambda + 2}}{\Gamma(\lambda + 2)}}
					\int_{\R_+} \varphi(\xi) \xi^{\frac{1}{2}(\lambda + 1)} 
					e^{-\xi (z - it)} \dif \xi  \\
			&= \sqrt{\frac{2^{\lambda + 2}}{\Gamma(\lambda + 2)}}
					\int_{\R_+} e^{it\xi}\varphi(\xi) \xi^{\frac{1}{2}(\lambda + 1)} 
					e^{-\xi z} \dif \xi \\
			&= \big(\cB_\lambda \circ \rho_t(\varphi)\big)(z),
	\end{align*}
	for every $z \in \C_+$.
\end{proof}

Using the properties of Bargmann transform $\cB_\lambda$ established so far, we can now describe the translation invariant subspaces of $\cA^2_\lambda(\C_+)$.

\begin{theorem}\label{thm:cA-invariant-subspaces}
	A Hilbert subspace $\cH \subset \cA^2_\lambda(\C_+)$ is translation invariant (see~Definition~\ref{def:invariance}) if and only if there exists a measurable subset $E \subset \R_+$ such that
	\[
		\cH = \cB_\lambda\big(L^2(E)\big).
	\]
	Furthermore, this identity establishes a one-to-one correspondence between measurable subsets (up to measure zero sets) of $\R_+$ and translation invariant Hilbert subspaces of $\cA^2_\lambda(\C_+)$.
\end{theorem}
\begin{proof}
	A well-known proof of the description of translation invariant Hilbert subspaces of $L^2(\R)$ shows that the Hilbert subspaces of $L^2(\R)$ invariant under $\rho$ (i.e.~character multiplication) are precisely those of the form $L^2(E)$ for $E \subset \R$ measurable (see~\cite[subsection~9.16]{RudinRCAnalysis}). This result holds for $L^2(\R_+)$ as well by the $\rho$-invariance of this subspace. Since $\cB_\lambda$ is unitary, the intertwining property established in Corollary~\ref{cor:Bargmann-intertwining} completes the proof.
\end{proof}

\begin{corollary}\label{cor:cA-invariant-subspaces}
	For every translation invariant subspace $\cH$ of $\cA^2_\lambda(\C_+)$ there exists a measurable subset $E \subset \R_+$, unique up to measure zero sets, such that $\cH$ is the closure of the space of functions of the form
	\[
		z \longmapsto 
			\int_E \varphi(\xi) \xi^{\frac{1}{2}(\lambda + 1)} 
				e^{-\xi z} \dif \xi
	\]
	where $\varphi \in C_c(\R_+) \chi_E = \{ g\chi_E: g \in C_c(\R_+) \}$.
\end{corollary}
\begin{proof}
	For a measurable subset $E \subset \R_+$, we note that $C_c(\R_+)\chi_E$ is dense in $L^2(E)$ and the formula for $\cB_\lambda$ from Corollary~\ref{cor:Bargmann-formulas} applies for functions in the subspace $C_c(\R_+)\chi_E$; the reason is that, for $\varphi \in C_c(\R_+)$ the Fourier transform in the variable $\xi$ of $V_\lambda(\varphi)$ (where $w = x + i\xi \in \C_+$) can be expressed using a Fourier integral. Hence, the result is a consequence of Theorem~\ref{thm:cA-invariant-subspaces}.
\end{proof}

\begin{remark}\label{rmk:cA-invariant-subspaces}
	We observe that the conclusion of Corollary~\ref{cor:cA-invariant-subspaces} still holds if we replace $C_c(\R_+)\chi_E$ with some other dense subspace of $L^2(E)$ whose elements $\varphi$ are such that $V_\lambda(\varphi)$ has a Fourier transform in the variable $\xi$ (where $w = x + i\xi \in \C_+$) that can be computed through the usual Fourier integral formula. We also note that, in all such cases, for any translation invariant subspace $\cH$ of $\cA^2_\lambda(\C_+)$ there exists a measurable subset $E \subset \R_+$ such that $\cH$ is the closure of the image of the integral operator that integrates only over $E$ the formula for $\cB_\lambda$ found in Corollary~\ref{cor:Bargmann-formulas}.
\end{remark}

\section{Translation and character invariant operators}
\label{sec:TranslationInvariantOperators}
The notion introduced in Definition~\ref{def:invariance} for subspaces has the next analogue for operators. Also, recall Remark~\ref{rmk:L2R-invariance} and the observations of the first paragraph of section~\ref{sec:BargmannIntertwining}.

\begin{definition}\label{def:Operatorinvariance}
	A bounded operator $T$ of either $\cA^2_\lambda(\C_+)$ or $L^2(\C_+, v_\lambda)$ will be called \textbf{translation invariant} when it intertwines the unitary representation $\tau$, in other words, if
	\[
		\tau_t \circ T = T \circ \tau_t,
	\]
	for every $t \in \R$. Correspondingly, a bounded operator $T$ on either $L^2(\R_+)$, $L^2(\R)$ or $L^2(\C_+, v_\lambda)$ will be called \textbf{character invariant} when it intertwines the unitary representation $\rho$, in other words, if
	\[
		\rho_t \circ T = T \circ \rho_t,
	\]
	for every $t \in \R$. We will denote with $\mfT(\cH)$ and $\mfC(\cH)$ the translation invariant and character invariant bounded operators, respectively, of $\cH$, for the corresponding cases of the Hilbert space $\cH$ considered in this definition.
\end{definition}

\begin{remark}\label{rmk:OperatorInvariance}
	The spaces $\mfT(\cH)$ and $\mfC(\cH)$, considered according to the choice of the Hilbert space $\cH$, are clearly von Neumann subalgebras of the algebra of bounded operators of $\cH$. The reason is that they are defined as spaces of intertwining operators of unitary representations.
\end{remark}

The following result will be useful latter on, but it is interesting by its own right.

\begin{proposition}\label{prop:ToepOpInvariance}
	If $A \in \mfT\big(L^2(\C_+, v_\lambda)\big)$, i.e.~it is a translation invariant bounded operator, then the Toeplitz operator $T^{(\lambda)}_A$ belongs to $\mfT\big(\cA^2_\lambda(\C_+)\big)$, i.e.~it is translation invariant as well.
\end{proposition}
\begin{proof}
	By Lemma~\ref{lem:cA-L2R-invariance}, the space $\cA^2_\lambda(\C_+)$ is a translation invariant subspace of $L^2(\C_+, v_\lambda)$. This implies that both $P_\lambda$, the Bergman projection, and $P_\lambda^*$, the natural inclusion, intertwine the unitary representations by translations defined on $L^2(\C_+, v_\lambda)$ and $\cA^2_\lambda(\C_+)$. If $A$ is translation invariant, then it intertwines translations as well. Hence, the conclusion follows from the identity $T^{(\lambda)}_A = P_\lambda \circ A \circ P_\lambda^*$.
\end{proof}

The description of translation invariant operators of $L^2(\R)$ is very well known (see~\cite{HormanderInvariant}), from which the description of character invariant operators of either $L^2(\R)$ or $L^2(\R_+)$ may be obtained. We present such description for the latter in the next result and provide a short proof for the sake of completeness. Similar conclusions are obtained in \cite[Section~2]{HMVVertical}.

\begin{proposition}\label{prop:CharInvariance}
	A bounded operator $T$ of $L^2(\R_+)$ is character invariant, i.e.~it belongs to $\mfT\big(L^2(\R_+)\big)$, if and only if there exists $h \in L^\infty(\R_+)$ such that $T = M_h$.
\end{proposition}
\begin{proof}
	First, note that for every $h \in L^\infty(\R_+)$ the operator $M_h$ is character invariant. The reason is that $\rho_t = M_{e^{it\xi}}$.
	
	Let us now consider a character invariant bounded operator $T$ of $L^2(\R_+)$. Then, its adjoint and so its real and imaginary parts are character invariant as well. Hence, we may assume that $T$ is self-adjoint. We consider the spectral resolution of $T$, which is given by a family $\cP$, consisting of orthogonal projections onto Hilbert subspaces of $L^2(\R_+)$, that satisfies the following properties.
	\begin{enumerate}
		\item The operator $T$ is the norm limit of a sequence of finite linear combinations of elements in $\cP$ (see~\cite[Theorem~5.2.2]{KRvolI}).
		\item The character invariance of the operator $T$ implies that every element of $\cP$ is character invariant (see~\cite[Theorem~1.44]{FollandHarmonic}).
	\end{enumerate}
	From (2) it follows that for every $P \in \cP$, the Hilbert subspace $P(L^2(\R_+))$ is character invariant, and so equals $L^2(E)$, for some Borel subset $E$ of $\R_+$, (see the proof of Theorem~\ref{thm:cA-invariant-subspaces}), whose orthogonal projection is $M_{\chi_E}$. We conclude that every element of $\cP$ is of the form $M_{\chi_E}$, for some measurable $E \subset \R_+$. Once this has been established, (1) implies that there exists a sequence $(h_n)_n \subset L^\infty(\R_+)$ of simple functions such that $\|T - M_{h_n}\| \to 0$ as $n \to +\infty$. Hence, $(h_n)_n$ converges to some $h \in L^\infty(\R_+)$ and we necessarily have $T = M_h$.
\end{proof}

Through the Bargmann transform $\cB_\lambda : L^2(\R_+) \rightarrow \cA^2_\lambda(\C_+)$, we obtain a description of the translation invariant operators of $\cA^2_\lambda(\C_+)$. It is worthwhile to compare the next result with \cite[Theorem~2.3]{HMVVertical}.

\begin{proposition}\label{prop:TranslationInvariance}
	A bounded operator $T$ of $\cA^2_\lambda(\C_+)$ is translation invariant, i.e.~it belongs to $\mfT\big(\cA^2_\lambda(\C_+)\big)$, if and only if there exists $h \in L^\infty(\R_+)$ such that $T = \cB_\lambda \circ M_h \circ \cB_\lambda^{-1}$.
\end{proposition}
\begin{proof}
	This follows immediately from Corollary~\ref{cor:Bargmann-intertwining} and Proposition~\ref{prop:CharInvariance}.
\end{proof}

\begin{definition}\label{def:spectralFromBargmann}
	For every $\lambda > -1$ and $T \in \mfT\big(\cA^2_\lambda(\C_+)\big)$, the unique function $h \in L^\infty(\R_+)$ such that $T = \cB_\lambda \circ M_h \circ \cB_\lambda^{-1}$, as obtained in Proposition~\ref{prop:TranslationInvariance}, will be called the \textbf{Bargmann spectral function of $T$}.
\end{definition}

\begin{remark}\label{rmk:TranslationInvariance}
	Proposition~\ref{prop:TranslationInvariance} implies that, up to $\cB_\lambda$, every translation invariant operator is a multiplier operator $M_h$ of $L^2(\R_+)$ for some $h \in L^\infty(\R_+)$. Furthermore, such function $h$ yields all the spectral properties of the translation invariant operator in question. This highlights the relevance of Bargmann transform in the study of translation invariant operators of $\cA^2_\lambda(\C_+)$. 
	
	Our next goal is to consider some special families of translation invariant operators and, in some cases, compute the corresponding function $h \in L^\infty(\R_+)$.
\end{remark}

In the rest of this work, we will use the notation introduced in the next definition.

\begin{definition}\label{def:hextensions}
	For a function $h \in L^\infty(\R_+)$ we define its extension by $0$ to $\R$ as the function $h_{(0)} \in L^\infty(\R)$ given by
	\begin{equation}\label{eq:h_(0)}
		h_{(0)}(\xi) = 
		\begin{cases}
			h(\xi), & \text{ if } \xi \in \R_+, \\
			0, & \text{ otherwise.}
		\end{cases}
	\end{equation}
	From this, we introduce the extension, constant on the real part, as the function $\widetilde{h} \in L^\infty(\C_+)$ given by
	\begin{equation}\label{eq:widetildeh}
		\widetilde{h}(x + i\xi) = h_{(0)}(\xi) =
		\begin{cases}
			h(\xi), & \text{ if } \xi \in \R_+, \\
			0, & \text{ otherwise.}
		\end{cases}
	\end{equation}	
\end{definition}

The next result will be useful to tackle the problem formulated in Remark~\ref{rmk:TranslationInvariance}.

\begin{lemma}\label{lem:Mwidetildeh}
	With the notation from Definition~\ref{def:hextensions} and for every $h \in L^\infty(\R_+)$, the following properties hold.
	\begin{enumerate}
		\item The function $\widetilde{h}$ belongs to $L^\infty(\C_+)$ and so both operators $M_h$ and $M_{\widetilde{h}}$ of $L^2(\R_+)$ and $L^2(\C_+, v_\lambda)$, respectively, are bounded operators.
		\item The intertwining properties
		\[
			M_{\widetilde{h}} \circ V_\lambda = V_\lambda \circ M_h,
			\qquad M_{h} \circ V_\lambda^* = V_\lambda^* \circ M_{\widetilde{h}},
		\]
		hold.
	\end{enumerate}	
\end{lemma}
\begin{proof}
	Property (1) es elementary, while property (2) follows from a straightforward computation using the expressions for $V_\lambda$ and $V_\lambda^*$ obtained in Proposition~\ref{prop:Vlambda}.
\end{proof}

We now prove that every translation invariant operator of $\cA^2_\lambda(\C_+)$ is a Toeplitz operator in the generalized sense considered in section~\ref{sec:BergmanToeplitz}. We recall that Definition~\ref{def:Operatorinvariance} introduced the notion of translation invariant operators for both $\cA^2_\lambda(\C_+)$ and $L^2(\C_+, v_\lambda)$, which are von Neumann algebras (see~Remark~\ref{rmk:OperatorInvariance}) denoted by $\mfT\big(\cA^2_\lambda(\C_+)\big)$ and $\mfT\big(L^2(\C_+, v_\lambda)\big)$, respectively.

\begin{theorem}\label{thm:TranslationInvariantOperatorsVonNeumann}
	For all $\lambda > -1$, every translation invariant operator of $\cA^2_\lambda(\C_+)$ is a Toeplitz operator whose symbol is a translation invariant operator of $L^2(\C_+, v_\lambda)$. Furthermore, let us define the set of operators of $L^2(\C_+, v_\lambda)$ given by
	\[
		\mfA_\lambda = \{ A_{\widetilde{h}} : h \in L^\infty(\R_+) \},
	\]
	where we denote $A_{\widetilde{h}} = \cF_\lambda \circ M_{\widetilde{h}} \circ \cF_\lambda^{-1}$ and $\widetilde{h}$ is defined by \eqref{eq:widetildeh}. Then, $\mfA_\lambda$ is a commutative von Neumann subalgebra of $\mfT\big(L^2(\C_+, v_\lambda)\big)$ that satisfies the following properties.
	\begin{enumerate}
		\item The assignment 
		\[
			h \mapsto A_{\widetilde{h}},
		\]
		yields an $*$-isomorphism of $*$-algebras from $L^\infty(\R_+)$ onto $\mfA_\lambda$.
		\item The assignments 
		\[
			h \mapsto T^{(\lambda)}_{A_{\widetilde{h}}}, \quad 
			A \mapsto T^{(\lambda)}_A
		\]
		define $*$-algebra isomorphisms from $L^\infty(\R_+)$ and $\mfA_\lambda$, respectively, onto $\mfT\big(\cA^2_\lambda(\C_+)\big)$.
		\item For every $T \in \mfT\big(\cA^2_\lambda(\C_+)\big)$, if we consider $h \in L^\infty(\R_+)$ the unique function such that $T = \cB_\lambda \circ M_h \circ \cB_\lambda^{-1}$, then $T = T^{(\lambda)}_{A_{\widetilde{h}}}$, for the bounded operator $A_{\widetilde{h}} = \cF_\lambda \circ M_{\widetilde{h}} \circ \cF_\lambda^{-1} \in \mfA_\lambda$.
	\end{enumerate}
\end{theorem}
\begin{proof}
	The space of operators $\{M_h : h \in L^\infty(\R_+)\}$ of $L^2(\R_+)$ is well known to be a von Neumann algebra isomorphic to $L^\infty(\R_+)$. This implies that the space of operators $\{M_{\widetilde{h}} : h \in L^\infty(\R_+)\}$ of $L^2(\C_+, v_\lambda)$ is a von Neumann algebra as well and it is isomorphic to $L^\infty(\R_+)$ (see~\cite[Example~11.2.1]{KRvolII}). Since, $\cF_\lambda$ is unitary, $\mfA_\lambda$ is a commutative von Neumann algebra of operators of $L^2(\C_+, v_\lambda)$. 

	Also, every operator $M_{\widetilde{h}}$, with $h \in L^\infty(\R_+)$, is clearly character invariant. Hence, Lemma~\ref{lem:Flambda-intertwining} implies that the elements of $\mfA_\lambda$ are translation invariant. The current discussion thus implies that $\mfA_\lambda$ is a commutative von Neumann subalgebra of $\mfT\big(L^2(\C_+, v_\lambda)\big)$. This establishes the first part of the statement.
	
	On the other hand, the assignments
	\begin{equation}\label{eq:h-to-Ah}
		h \mapsto \widetilde{h} \mapsto A_{\widetilde{h}},
	\end{equation}
	from $L^\infty(\R_+)$ to $L^\infty(\C_+)$ and then to $\mfA_\lambda$ are clearly injective homomorphisms of $*$-algebras. This uses the definition of $\widetilde{h}$ and the fact that $\cF_\lambda$ is unitary. Since the image of the composition of these maps is $\mfA_\lambda$, this proves (1).

	Let us consider $T \in \mfT\big(\cA^2_\lambda(\C_+)\big)$, a translation invariant bounded operator of $\cA^2_\lambda(\C_+)$. First, we construct the following diagrams and prove their commutativity.
	\begin{equation}\label{eq:Diagrams}
		\begin{tikzcd}
			L^2(\C_+. v_\lambda) \arrow[r, "M_{\widetilde{h}}"] &
			L^2(\C_+, v_\lambda) 
					\arrow[d, "V_\lambda^*"']
				& L^2(\C_+, v_\lambda) \arrow[r, "A_{\widetilde{h}}"]
					&
			L^2(\C_+, v_\lambda) \arrow[d, "P_\lambda"']  \\
			L^2(\R_+) \arrow[r, "M_h"] \arrow[u, hook', "V_\lambda"']  
				&L^2(\R_+) 
				& \cA^2_\lambda(\C_+) \arrow[u, hook'] 
					\arrow[r, "T"] 
				& \cA^2_\lambda(\C_+)
		\end{tikzcd}
	\end{equation}
	By Proposition~\ref{prop:TranslationInvariance}, there exists a unique $h \in L^\infty(\R_+)$ such that $T = \cB_\lambda \circ M_h \circ \cB_\lambda^{-1}$. From this, Lemma~\ref{lem:Mwidetildeh} implies the existence of $\widetilde{h} \in L^\infty(\C_+)$ satisfying \eqref{eq:widetildeh}. Then, we consider the operator of $L^2(\C_+, v_\lambda)$ defined by
	\[
		A_{\widetilde{h}} = \cF_\lambda \circ M_{\widetilde{h}} \circ \cF_\lambda^{-1},
	\]
	which belongs to $\mfA_\lambda$ by definition. Together with the maps $V_\lambda, V_\lambda^*$, from Proposition~\ref{prop:Vlambda}, and the Bergman projection $P_\lambda$, this yields the arrows in both diagrams \eqref{eq:Diagrams}. We now establish their commutativity.
	
	We compute
	\begin{equation}\label{eq:Mh-Vlambda}
		V_\lambda^* \circ M_{\widetilde{h}} \circ V_\lambda 
			= V_\lambda^* \circ V_\lambda \circ M_h = M_h,
	\end{equation}
	where the first identity follows from Lemma~\ref{lem:Mwidetildeh}(2) and the second identity follows from the fact that $V_\lambda$ is an isometric inclusion and $V_\lambda^*$ is thus the corresponding projection, so that $V_\lambda^* \circ V_\lambda$ is the identity in $L^2(\R_+)$. Hence, the left diagram of \eqref{eq:Diagrams} is commutative.
	
	By Theorem~\ref{thm:Bargmann-diagram}, the unitary maps $\cF_\lambda$ and $\cB_\lambda$ yield a commutative diagram with the vertical arrows of the diagrams in \eqref{eq:Diagrams}. Hence, in \eqref{eq:Diagrams}, the commutativity of the left diagram implies the commutativity of the right diagram. We prove this explicitly for the sake of completeness. 
	
	Theorem~\ref{thm:Bargmann-diagram} yields the identity
	\begin{equation}\label{eq:Plambda-Vlambda}
		P_\lambda = \cB_\lambda \circ V_\lambda^* \circ \cF_\lambda^{-1},
	\end{equation}
	which allows us to compute as follows
	\begin{align*}
		P_\lambda \circ A_{\widetilde{h}} \circ P_\lambda^* 
			&= \big(\cB_\lambda \circ V_\lambda^* \circ \cF_\lambda^{-1}\big)
				\circ \big(\cF_\lambda \circ M_{\widetilde{h}} \circ \cF_\lambda^{-1}\big) 
				\circ 
				\big(\cB_\lambda \circ V_\lambda^* \circ \cF_\lambda^{-1}\big)^* \\
			&= \cB_\lambda \circ V_\lambda^* \circ M_{\widetilde{h}} 
					\circ V_\lambda \circ \cB_\lambda^{-1}  \\
			&= \cB_\lambda \circ M_h \circ \cB_\lambda^{-1} = T,
	\end{align*}
	where the first identity uses \eqref{eq:Plambda-Vlambda}, the third identity uses \eqref{eq:Mh-Vlambda} and the last identity follows from the definition of $h$. Since $P_\lambda^*$ is the natural inclusion from $\cA^2_\lambda(\C_+)$ into $L^2(\C_+, v_\lambda)$ (because $P_\lambda$ is the Bergman projection), this proves that the right diagram in \eqref{eq:Diagrams} is indeed commutative. By the definition of Toeplitz operators, we thus have
	\[
		T = P_\lambda \circ A_{\widetilde{h}} \circ P_\lambda^* 
				= T^{(\lambda)}_{A_{\widetilde{h}}}.
	\]
	Since $T \in \mfT\big(\cA^2_\lambda(\C_+)\big)$ was arbitrarily given, this proves (3). It also establishes the surjectivity of the assignments in (2). Note that it is straightforward to prove the linearity of such assignments.
	
	We will now prove the injectivity of the assignments in (2). Let us assume that $T^{(\lambda)}_A = 0$, with $A \in \mfA_\lambda$. Then, there exists $h \in L^\infty(\R_+)$ such that $A = A_{\widetilde{h}}$, for which we have
	\[
		P_\lambda \circ A_{\widetilde{h}} \circ P_\lambda^* =  T^{(\lambda)}_{A_{\widetilde{h}}} = T^{(\lambda)}_A = 0,
	\]
	where we used the definition of Toeplitz operator in the first identity. On the other hand, this last computation, the definition of $A_{\widetilde{h}}$ and \eqref{eq:Plambda-Vlambda}, together with the commutative diagram arguments used before, lead us to
	\[
		M_h = V_\lambda^* \circ M_{\widetilde{h}} \circ V_\lambda = 0,
	\]
	where the first identity follows from \eqref{eq:Mh-Vlambda}. This implies that $h = 0$ and so that $\widetilde{h} = 0$. We conclude that $A = A_{\widetilde{h}} = 0$. This proves the injectivity of the assignments in (2) and thus that they define linear isomorphisms.
	
	Finally, we prove that the assignments in (2) are $*$-isomorphisms. Note that we have constructed the diagrams \eqref{eq:Diagrams} so that their commutativity holds with $T = T^{(\lambda)}_{A_{\widetilde{h}}}$. Hence, the first assignment in (2) is a $*$-isomorphism because it is the inverse of the map $T \mapsto h$, where $h$ is given by the condition $T = \cB_\lambda \circ M_h \circ \cB_\lambda^{-1}$, which is clearly a $*$-isomorphism. From the definition of $\mfA_\lambda$ and since the maps in \eqref{eq:h-to-Ah} are $*$-isomorphisms, the previous case shows that the second assignment in (2) is a $*$-isomorphism as well. This completes the proof of (2) and of the Theorem.
\end{proof}

\begin{remark}\label{rmk:TranslationInvariantOperatorsVonNeumann}
	As it was first proved in \cite{HHMVerticalWeighted} (see~also~\cite{HMVVertical}), for every $\lambda > -1$, the $C^*$-algebra generated by Toeplitz operators $T^{(\lambda)}_a$ with translation invariant symbols $a \in L^\infty(\C_+)$ consists precisely of those translation invariant operators $T \in \mfT\big(\cA^2_\lambda(\C_+)\big)$ for which the function $h \in L^\infty(\R_+)$ given by
	\begin{equation}\label{eq:TIOVN}
		M_h = \cB_\lambda^{-1} \circ T \circ \cB_\lambda
	\end{equation}
	is such that $h \circ \exp : \R \rightarrow \C$ is uniformly continuous; for an explanation of this fact, we refer to the last paragraph of Remark~\ref{rmk:spectral-FourierIntertwining}. On the other hand we have proved that, under the correspondence given by \eqref{eq:TIOVN}, the translation invariant bounded operators correspond to the full space $L^\infty(\R_+)$ (see~Proposition~\ref{prop:TranslationInvariance}). In particular, the von Neumann algebra $\mfT\big(\cA^2_\lambda(\C_+)\big)$ is much larger than the $C^*$-algebra generated by Toeplitz operators $T^{(\lambda)}_a$, with $a \in L^\infty(\C_+)$ translation invariant.
	
	In contrast with the conclusion of the previous paragraph, we have proved in Theorem~\ref{thm:TranslationInvariantOperatorsVonNeumann} that every bounded operator $T \in \mfT\big(\cA^2_\lambda(\C_+)\big)$ can be written $T = T^{(\lambda)}_A$ for some $A \in \mfT\big(L^2(\C_+, v_\lambda)\big)$, and so that every such $T$ is itself a Toeplitz operator whose symbol is an operator. This conclusion alone is actually not very difficult to prove when no further restriction on $A$, besides translation invariance, is required. With this respect, the non-triviality of Theorem~\ref{thm:TranslationInvariantOperatorsVonNeumann} lies in the fact that we have achieved this by singling out the von Neumann algebra $\mfA_\lambda$ from which a unique choice of $A \in \mfA_\lambda$ such that $T = T^{(\lambda)}_A$ is possible, for every $T \in \mfT\big(\cA^2_\lambda(\C_+)\big)$. Furthermore, our particular choice of the von Neumann algebra $\mfA_\lambda$ is non-trivial because it is naturally constructed from the Bargmann spectral functions $h \in L^\infty(\R_+)$ of operators $T \in \mfT\big(\cA^2_\lambda(\C_+)\big)$ as established in Theorem~\ref{thm:TranslationInvariantOperatorsVonNeumann}(3).
	
	Another non-trivial feature of our construction is that the assignment $A \mapsto T^{(\lambda)}_A$ is a $*$-isomorphism, as a map from $\mfA_\lambda$ onto $\mfT\big(\cA^2_\lambda(\C_+)\big)$. To understand this, recall that the assignment $a \mapsto T^{(\lambda)}_a$, for function symbols $a \in L^\infty(\C_+)$, does not preserve product, even after restriction to translation invariant functions. However, it does preserve product when the assignment, as in Theorem~\ref{thm:TranslationInvariantOperatorsVonNeumann}(2), is considered for operator symbols $A \in \mfA_\lambda$.
	
	The remark in the previous paragraph is also related to the non-triviality of the uniqueness of choice of $A \in \mfA_\lambda$ so that $T = T^{(\lambda)}_A$, for a given $T \in \mfT\big(\cA^2_\lambda(\C_+)\big)$. To see the non-triviality of the uniqueness of $A$, we recall the well-known fact that, for symbols $a \in L^\infty(\C_+)$, the identity $T^{(\lambda)}_a = 0$, for some $\lambda > -1$, implies $a = 0$. However, this does not hold for Toeplitz operators of the general type considered here. More precisely, there exist non-zero bounded operators $A$ of $L^2(\C_+, v_\lambda)$ for which $T^{(\lambda)}_A = 0$. Of course, by Theorem~\ref{thm:TranslationInvariantOperatorsVonNeumann}, this cannot happen when $A \in \mfA_\lambda$.
\end{remark}

Recall that a bounded operator of $L^2(\R)$ is called a convolution type operator when it is of the form $\cF \circ M_g \circ \cF^{-1}$ for some $g \in L^\infty(\R)$.

\begin{theorem}\label{thm:AlambdaConvolutionType}
	With the notation of Theorem~\ref{thm:TranslationInvariantOperatorsVonNeumann} and \eqref{eq:h_(0)}, every operator $A_{\widetilde{h}}$ belonging to $\mfA_\lambda$, where $h \in L^\infty(\R_+)$, satisfies
	\[
		A_{\widetilde{h}} = I \otimes 
				\big( \cF \circ M_{h_{(0)}} \circ \cF^{-1}\big),
	\]
	with respect to the tensor product decomposition $L^2(\C_+, v_\lambda) = L^2(\R_+, \widehat{v}_\lambda) \otimes L^2(\R)$ and where $\cF \circ M_{h_{(0)}} \circ \cF^{-1}$ is a convolution type operator. Furthermore, the von Neumann algebra $\mfT\big(\cA^2_\lambda(\C_+)\big)$ of translation invariant bounded operators of $\cA^2_\lambda(\C_+)$ consists of Toeplitz operators of the form
	\begin{equation}\label{eq:IotimesC}
		T = T^{(\lambda)}_{I \otimes C},
	\end{equation}
	where $C$ is a convolution type operator of $L^2(\R)$.
\end{theorem}
\begin{proof}
	For $f \in L^2(\R_+, \widehat{v}_\lambda)$ and $g \in L^2(\R)$, we have
	\begin{align*}
		\big(M_{\widetilde{h}}(f \otimes g)\big)(x + i\xi)
		&= \widetilde{h}(x + i\xi) f(x) g(\xi) \\
		&= f(x) h_{(0)}(\xi) g(\xi) \\
		&= \big((I \otimes M_{h_{(0)}})(f \otimes g)\big)(x + i\xi),
	\end{align*}
	for almost every $x + i\xi \in \C_+$. By the density of the elements $f \otimes g$ considered, this yields
	\[
		M_{\widetilde{h}} = I \otimes M_{h_{(0)}},
	\]
	which implies, by the definitions of $A_{\widetilde{h}}$ and $\cF_\lambda$, the first identity in the statement.
	
	To prove the second claim, we note that the previous paragraph and Theorem~\ref{thm:TranslationInvariantOperatorsVonNeumann}(3) imply that every element of $\mfT\big(\cA^2_\lambda(\C_+)\big)$ is of the form
	\[
		T^{(\lambda)}_{A_{\widetilde{h}}} = T^{(\lambda)}_{I \otimes C},
	\]
	where $C = \cF \circ M_{h_{(0)}} \circ \cF^{-1}$ is a convolution type operator of $L^2(\R)$. Conversely, let $C = \cF \circ M_g \circ \cF^{-1}$ be a convolution type operator of $L^2(\R)$, where $g \in L^\infty(\R)$. We need to prove that $T^{(\lambda)}_{I \otimes C}$ belongs to $\mfT\big(\cA^2_\lambda(\C_+)\big)$, i.e.~it is translation invariant. To achieve this, we note that, for every $t \in \R$, the translation operator on $L^2(\C_+, v_\lambda)$ satisfies
	\begin{equation}\label{eq:tautensor}
		\tau_t = I \otimes \hat{\tau}_t
	\end{equation}
	with respect to the tensor product decomposition $L^2(\C_+, v_\lambda) = L^2(\R_+, \widehat{v}_\lambda) \otimes L^2(\R)$ and where $\hat{\tau}_t$ is the translation operator
	\[
		\big(\hat{\tau}_t(g)\big)(y) = g(y - t),
	\]
	defined for $g \in L^2(\R)$ and $y \in \R$. On the other hand, Lemma~\ref{lem:cA-L2R-invariance} implies that $P_\lambda$ and $P_\lambda^*$ intertwine $\tau_t$. Hence, using \eqref{eq:tautensor} we compute
	\begin{align*}
		\tau_t^{-1} \circ T^{(\lambda)}_{I \otimes C} \circ \tau_t
			&= \tau_t^{-1} \circ P_\lambda \circ (I \otimes C) \circ P_\lambda^* 
					\circ \tau_t  \\
			&= P_\lambda \circ \tau_t^{-1} \circ (I \otimes C) \circ \tau_t 
				\circ P_\lambda^*   \\
			&= P_\lambda \circ \big( I 
					\otimes (\hat{\tau}_t^{-1} \circ C \circ \hat{\tau}_t) \big)
				\circ P_\lambda^*  \\	
			&= P_\lambda \circ (I \otimes C) \circ P_\lambda^* 
				= T^{(\lambda)}_{I \otimes C},
	\end{align*}
	where we have used in the second to last identity that convolution type operators of $L^2(\R)$ commute with the translations $\hat{\tau}_t$. We conclude that $T^{(\lambda)}_{I \otimes C}$ does belong to $\mfT\big(\cA^2_\lambda(\C_+)\big)$. This completes the proof.
\end{proof}

\begin{remark}\label{rmk:AlambdaConvolutionType}
	Theorem~\ref{thm:AlambdaConvolutionType} allows to more precisely understand the structure of the von Neumann algebra $\mfA_\lambda$ introduced in Theorem~\ref{thm:TranslationInvariantOperatorsVonNeumann}: it consists of bounded operators that act on $L^2(\C_+, v_\lambda) = L^2(\R_+, \widehat{v}_\lambda) \otimes L^2(\R)$ through a convolution type operator on the factor $L^2(\R)$, and precisely with functions $h_{(0)}$, which are the extension by zero of $h \in L^\infty(\R_+)$, i.e.~ all possible Bargmann spectral functions $h$ of the operators belonging to $\mfT\big(\cA^2_\lambda(\C_+)\big)$.
	
	Through Theorem~\ref{thm:AlambdaConvolutionType} we also described $\mfT\big(\cA^2_\lambda(\C_+)\big)$ as the Toeplitz operators associated to $I\otimes C$ where $C$ runs through all convolution type operators of $L^2(\R)$.
\end{remark}

\section{Fourier intertwining and imaginary part independent operators}
\label{sec:MoreTranslationInvOp}
In the previous section, we have considered operators of the form $I \otimes C$, with respect to the tensor product decomposition $L^2(\C_+, v_\lambda) = L^2(\R_+, \widehat{v}_\lambda) \otimes L^2(\R)$, where $C$ is a convolution type operator. More precisely, these operators can be written as
\[
	I \otimes C = I \otimes (\cF \circ M_{h_{(0)}} \circ \cF^{-1}) 
		= \cF_\lambda \circ (I \otimes M_{h_{(0)}}) \circ \cF_\lambda^{-1},
\]
where $\cF_\lambda = I \otimes \cF$ (see~section~\ref{sec:Bargmann}) and $h_{(0)}$ is the extension by $0$ of a given $h \in L^\infty(\R_+)$ (see~Definition~\ref{def:hextensions}). The relevant properties of these operators in our study of Toeplitz operators was established in Theorems~\ref{thm:TranslationInvariantOperatorsVonNeumann} and \ref{thm:AlambdaConvolutionType}.

We will now consider a sort of complementary type of operators: those that commute with $\cF_\lambda$, instead being obtained from conjugation by $\cF_\lambda$.

\begin{definition}\label{def:Flambda-intertwining}
	A bounded operator $A$ of $L^2(\C_+, v_\lambda)$ will be called \textbf{Fourier intertwining for $L^2(\C_+, v_\lambda)$}, or simply \textbf{$\cF_\lambda$-intertwining}, when $\cF_\lambda \circ A = A \circ \cF_\lambda$. In this case, the corresponding Toeplitz operator $T^{(\lambda)}_A$, acting on $\cA^2_\lambda(\C_+)$ (for $\lambda > -1$), will be said to have \textbf{Fourier intertwining operator symbol}.
\end{definition}

Another useful family of operators is introduced in the next definition. For this, we consider again the tensor product decomposition of $L^2(\C_+, v_\lambda)$.

\begin{definition}\label{def:ImPartInd}
	A bounded operator $A$ of $L^2(\C_+, v_\lambda)$ will be called \textbf{imaginary part independent} if and only if there is a bounded operator $\widehat{A}$ of $L^2(\R_+, \widehat{v}_\lambda)$ such that $A = \widehat{A} \otimes I$ with respect to the tensor product decomposition $L^2(\C_+, v_\lambda) = L^2(\R_+, \widehat{v}_\lambda) \otimes L^2(\R)$ (see~\eqref{eq:L2-tensordecomposition}). With this notation, the corresponding Toeplitz operator $T^{(\lambda)}_A$, acting on $\cA^2_\lambda(\C_+)$ (for $\lambda > -1$), will be said to have \textbf{imaginary part independent operator symbol}.
\end{definition}

We establish some first properties of this sort of operators. 

\begin{proposition}\label{prop:ImPartInd}
	Let $A = \widehat{A} \otimes I$ be an imaginary part independent operator of $L^2(\C_+,v_\lambda)$, where $\widehat{A}$ is a bounded operator of $L^2(\R_+, \widehat{v}_\lambda)$. Then, the following properties are satisfied.
	\begin{enumerate}
		\item The operator $A$ is $\cF_\lambda$-intertwining (see~Definition~\ref{def:Flambda-intertwining}) and translation invariant.
		\item For every $f \in L^2(\R_+, \widehat{v}_\lambda)$ and $g \in L^2(\R)$ we have
		\[
			\big(A(f \otimes g)\big)(x + iy) 
				= \big(\widehat{A}(f)\big)(x) g(y),
		\]
		for almost every $x + iy \in \C_+$.
	\end{enumerate}
	In particular, the Toeplitz operator $T^{(\lambda)}_A$ is translation invariant.
\end{proposition}
\begin{proof}
	Recall that $\cF_\lambda = I \otimes \cF$. Also, as obtained in the proof of Theorem~\ref{thm:AlambdaConvolutionType}, the translation operators of $L^2(\C_+, v_\lambda)$ satisfy, for every $t \in \R$, the identity
	\[
		\tau_t = I \otimes \hat{\tau}_t,
	\]
	where $\hat{\tau}_t$ is the corresponding translation operator of $L^2(\R)$ (see~\eqref{eq:tautensor}). Hence, the claims in (1) are immediate consequences of our choice for $A$.

	If we consider $f \in L^2(\R_+, \widehat{v}_\lambda)$ and $g \in L^2(\R)$, then we have $A(f\otimes g) = \big(\widehat{A}(f)\big) \otimes g$, and so evaluation at $x + iy \in \C_+$ yields (2).
	
	The last claim is a consequence of (1) and Proposition~\ref{prop:ToepOpInvariance}.
\end{proof}

The next result clarifies the meaning of the notions just introduced for the case of function symbols.

\begin{corollary}\label{cor:ImPartIndSymbols}
	Let $a \in L^\infty(\C_+)$ be given. Then, the following properties are equivalent.
	\begin{enumerate}
		\item The symbol $a$ is translation invariant.
		\item For some (and hence every) $\lambda > -1$, the operator $M_a$ acting on $L^2(\C_+, v_\lambda)$ is imaginary part independent.
	\end{enumerate}
\end{corollary}
\begin{proof}
	Assuming (1), the function $a$ is translation invariant, thus implying the existence of a function $\widehat{a} \in L^\infty(\R_+)$ such that
	\[
		a(x + iy) = \widehat{a}(x),
	\]
	for almost every $x + i y \in \C_+$. Hence, we have
	\[
		M_a = M_{\widehat{a}} \otimes I,
	\]
	which yields (2).
	
	Let us now assume that (2) holds. Then, there exists a bounded operator $\widehat{A}$ of $L^2(\R_+, \widehat{v}_\lambda)$ such that $M_a = \widehat{A} \otimes I$. Let us choose continuous everywhere positive functions $f \in L^2(\R_+, \widehat{v}_\lambda)$ and $g \in L^2(\R)$. Proposition~\ref{prop:ImPartInd}(2) implies
	\[
	a(x + iy)f(x)g(y) = \big(\widehat{A}(f)\big)(x) g(y),
	\]
	for almost every $x + iy \in \C_+$. Hence, our choice of $f,g$ leads to
	\[
	a(x + iy) = \frac{\big(\widehat{A}(f)\big)(x)}{f(x)},
	\]
	for almost every $x + iy \in \C_+$, as well. This clearly implies (1).
\end{proof}

\begin{remark}\label{rmk:ImPartInd}
	Corollary~\ref{cor:ImPartIndSymbols} has established that $a \in L^\infty(\C_+)$ is independent of its imaginary part variable if and only if the multiplier operator $M_a$ is imaginary part independent as introduced in Definition~\ref{def:ImPartInd}. This explains our choice of notation. 
	
	On the other hand, there are plenty of imaginary part independent operators that are not multipliers by functions: e.g.~any operator of the form $\widehat{A} \otimes I$, where $\widehat{A}$ is a bounded operator of $L^2(\R_+, \widehat{v}_\lambda)$ that is not a multiplier by a function. Note that this last claim follows from the proof of Corollary~\ref{cor:ImPartIndSymbols}. By Proposition~\ref{prop:ImPartInd}(1), we also conclude the existence of plenty $\cF_\lambda$-intertwining translation invariant operators beyond those of the form $M_a$, where $a \in L^\infty(\C_+)$ is translation invariant. 
\end{remark}

We obtain another property for imaginary part independent operators that further explains this notation. It will be useful in the computation of Bargmann spectral functions. 

\begin{proposition}\label{prop:ImPartInd-withfunction}
	Let $A$ be an imaginary part independent bounded operator of $L^2(\C_+, v_\lambda)$. Let $b \in L^\infty(\C_+)$ be a function that does not depend on its real part variable, i.e.~for which there exists $\widehat{b} \in L^\infty(\R)$ such that $b(x + i\xi) = \widehat{b}(\xi)$ for almost every $x + i\xi \in \C_+$. Then, we have
	\[
		A(b\psi) = b A(\psi),
	\]
	for every $\psi \in L^2(\C_+, v_\lambda)$.
\end{proposition}
\begin{proof}
	Let us write $A = \widehat{A} \otimes I$ for some bounded operator $\widehat{A}$ of $L^2(\R_+, \widehat{v}_\lambda)$. For $b$ and $\widehat{b}$ as in the statement, we have
	\[
		M_b = I \otimes M_{\widehat{b}}.
	\]
	Thus, the result follows from the fact that the operators $\widehat{A} \otimes I$ and $I \otimes M_{\widehat{b}}$ commute with each other.
\end{proof}

\section{Bargmann spectral functions of Toeplitz operators with Fourier intertwining operator symbols}
\label{sec:Bargmann-spectral}
In this section we compute the Bargmann spectral function for Toeplitz operators with translation invariant operator symbols of the type introduced in section~\ref{sec:MoreTranslationInvOp}. We start with Fourier intertwining operators.

\begin{theorem}\label{thm:spectral-FourierIntertwining}
	Let $A$ be a bounded operator of $L^2(\C_+, v_\lambda)$ which is $\cF_\lambda$-intertwining. If $A$ is also translation invariant, then the Bargmann spectral function of the Toeplitz operator $T^{(\lambda)}_A$ is the (unique) function $h_A : \R_+ \rightarrow \C$ that satisfies
	\begin{equation}\label{eq:spectral-FourierIntertwining}
		h_A \varphi = \big(V_\lambda^* \circ A \circ V_\lambda\big)(\varphi),
	\end{equation}
	for every $\varphi \in L^2(\R_+)$. In particular, for every such $\varphi$ we have
	\begin{equation}\label{eq:spectral-FourierIntegral}
		h_A(\xi) \varphi(\xi) 
			= \frac{(2\sqrt{\xi})^{\lambda + 1}}{\Gamma(\lambda + 1)}
						\int_{\R_+} \big(A(\Phi)\big)(x + i \xi)
							e^{-x\xi} x^\lambda \dif x,
	\end{equation}
	for almost every $\xi \in \R_+$, where the function $\Phi \in L^2(\C_+, v_\lambda)$ is defined by
	\[
		\Phi(x + i\xi) = \xi^{\frac{1}{2}(\lambda + 1)}
					\varphi(\xi) e^{-x\xi} \chi_{\R_+}(\xi),
	\]
	for almost every $x + i\xi \in \C_+$.
\end{theorem}
\begin{proof}
	The translation invariance of $A$ and Proposition~\ref{prop:ToepOpInvariance} imply that $T^{(\lambda)}_A$ is translation invariant as well. Hence, Proposition~\ref{prop:TranslationInvariance} implies the existence of a (unique) function $h_A \in L^\infty(\R_+)$ such that 
	\[
		M_{h_A} = \cB_\lambda^{-1} \circ T^{(\lambda)}_A \circ \cB_\lambda.
	\]
	We now proceed to compute the values of the latter.
	
	Recall, from Theorem~\ref{thm:Bargmann-diagram}, that $\cB_\lambda = \cF_\lambda \circ V_\lambda$, which yields
	\[
		\cB_\lambda^* \circ \cB_\lambda = V_\lambda^* \circ V_\lambda 
				= I_{L^2(\R_+)},
	\]
	the identity operator in $L^2(\R_+)$, because $V_\lambda$ and $V_\lambda^*$ are the inclusion and the 
	projection, respectively, of $L^2(\R_+)$ with respect to $L^2(\C_+, v_\lambda)$. Hence, we also have
	\[
	\cB_\lambda \circ \cB_\lambda^* 
	= \cF_\lambda \circ V_\lambda \circ V_\lambda^* \circ \cF_\lambda^{-1} 
	= P_\lambda,
	\]
	where we now use the commutative diagram \eqref{eq:Bargmann-diagram} and the fact that $\cF_\lambda$ is unitary. 
	
	The three displayed formulas above, and the definition of Toeplitz operator, allow us to compute as follows
	\begin{align*}
		M_{h_A} 
		&= \cB_\lambda^* \circ P_\lambda^* \circ A \circ P_\lambda \circ \cB_\lambda \\
		&= \cB_\lambda^* \circ \cB_\lambda \circ \cB_\lambda^*
		\circ A \circ \cB_\lambda \circ \cB_\lambda^* \circ \cB_\lambda \\
		&= \cB_\lambda^* \circ A \circ \cB_\lambda \\
		&= V_\lambda^* \circ \cF_\lambda^{-1} \circ A \circ \cF_\lambda \circ V_\lambda \\
		&= V_\lambda^* \circ A \circ V_\lambda,
	\end{align*}
	where we have used the $\cF_\lambda$-intertwining property of $A$ in the last identity. Evaluation at $\varphi \in L^2(\C_+, v_\lambda)$ yields \eqref{eq:spectral-FourierIntertwining}.
	
	Let us now fix $\varphi \in L^2(\R_+)$. Note that $\Phi$, as defined in the statement, is a constant multiple of $V_\lambda(\varphi)$, and so it does belong to $L^2(\C_+, v_\lambda)$. Using the expressions for $V_\lambda$ and $V_\lambda^*$ obtained in Proposition~\ref{prop:Vlambda} we compute, for almost every $\xi \in \R_+$, as follows
	\begin{align*}
		h_A(\xi) \varphi(\xi) 
			&= \sqrt{\frac{2^\lambda(\lambda + 1)\xi^{\lambda + 1}}{2\pi \Gamma(\lambda + 1)}}
					\sqrt{\frac{\pi 2^{\lambda + 3}}{\Gamma(\lambda + 2)}}
						\int_{\R_+} \big(A(\Phi)\big)(x + i\xi)
							e^{-x\xi} x^\lambda \dif x  \\
		&= \frac{(2\sqrt{\xi})^{\lambda + 1}}{\Gamma(\lambda + 1)}
				\int_{\R_+} \big(A(\Phi)\big)(x + i\xi)
							e^{-x\xi} x^\lambda \dif x,
	\end{align*}
	where the first identity uses the linearity of $A$ and that $\Phi$ and $V_\lambda(\varphi)$ differ by a constant factor. This yields \eqref{eq:spectral-FourierIntegral}.
\end{proof}

As a consequence, we obtain the Bargmann spectral formula for the case of a multiplier operator by a translation invariant function. This formula was already obtained in \cite{GKVParabolic} (see also \cite[Theorem~4.3.5]{BarreraQuirogaHeisenberg} and \cite{HHMVerticalWeighted}).

\begin{corollary}\label{cor:spectral-FourierIntertwining}
	Let $a \in L^\infty(\C_+)$ be a translation invariant symbol and $\widehat{a} \in L^\infty(\R_+)$ the function such that $a(x + i\xi) = \widehat{a}(x)$, for almost every $x + i\xi \in \C_+$. Then, the Bargmann spectral function $h_a \in L^\infty(\R_+)$ defined such that
	\[
		M_{h_a} = \cB_\lambda^{-1} \circ T^{(\lambda)}_a \circ \cB_\lambda,
	\]
	is given by
	\[
		h_a(\xi) = \frac{(2\xi)^{\lambda + 1}}{\Gamma(\lambda + 1)}
			\int_{\R_+} \widehat{a}(x) e^{-2x\xi} x^\lambda \dif x,
	\]
	for almost every $\xi \in \R_+$.
\end{corollary}
\begin{proof}
	By Corollary~\ref{cor:ImPartIndSymbols} and Proposition~\ref{prop:ImPartInd}(1), the operator $M_a$ is $\cF_\lambda$-intertwining and translation invariant. Hence, Theorem~\ref{thm:spectral-FourierIntertwining} can be applied to $M_a$. First we note that, with the notation of the latter, we have
	\[
		\big(M_a(\Phi)\big)(x + i\xi) 
			= \widehat{a}(x) \xi^{\frac{1}{2}(\lambda + 1)} \varphi(\xi) e^{-x\xi}
	\chi_{\R_+}(\xi),
	\]
	for every $\varphi \in L^2(\R_+)$ and almost every $x + i\xi \in \C_+$. The result now follows by substitution in \eqref{eq:spectral-FourierIntegral}.
\end{proof}

\begin{remark}\label{rmk:spectral-FourierIntertwining}
	Corollary~\ref{cor:spectral-FourierIntertwining} and its proof show that Theorem~\ref{thm:spectral-FourierIntertwining} extends the ``diagonalization'' property and corresponding spectral formula for Toeplitz operators from the already known case of translation invariant function symbols to the case of translation invariant operator symbols, when the latter are Fourier intertwining as introduced in Definition~\ref{def:Flambda-intertwining}.
	
	On the other hand, this generalization is achieved through formulas that enlighten the behavior of Toeplitz operators with translation invariant operator symbols as we now explain. The Toeplitz operator with operator symbol $A$ is defined~by
	\[
	T^{(\lambda)}_A = P_\lambda \circ A \circ P_\lambda^*,
	\]
	while the proof of Theorem~\ref{thm:spectral-FourierIntertwining} and its formula \eqref{eq:spectral-FourierIntertwining} have shown that $T^{(\lambda)}_A$ corresponds, up to the Bargmann transform $\cB_\lambda$, to the operator
	\[
	S^{(\lambda)}_A = V_\lambda^* \circ A \circ V_\lambda,
	\]
	when $A$ is Fourier intertwining and translation invariant. Since, $V_\lambda$ and $V_\lambda^*$ are the inclusion and projection of $L^2(\R_+)$ as a Hilbert subspace of $L^2(\C_+, v_\lambda)$, the operator $S^{(\lambda)}_A$ is itself a sort of Toeplitz operator for this pair of Hilbert spaces and with the same operator symbol $A$. Hence, for any Fourier intertwining translation invariant operator $A$, the Bargmann transform $\cB_\lambda$ translates the Toeplitz operator $T^{(\lambda)}_A$ of $\cA^2_\lambda(\C_+)$ into the Toeplitz operator $S^{(\lambda)}_A$ of $L^2(\R_+)$. Note that with the current choices, the latter is also the multiplier operator $M_{h_A}$ (see~\eqref{eq:spectral-FourierIntertwining}).
	
	Finally, we recall that the spectral formula obtained in Corollary~\ref{cor:spectral-FourierIntertwining} implies that the $C^*$-algebra generated by Toeplitz operators with translation invariant function symbols correspond, under the Bargmann transform, to the space of functions $h \in C(\R_+)$ such that $h \circ \exp$ is uniformly continuous. This last claim, within our current notation, is the content of \cite[Theorem~2]{HHMVerticalWeighted}.
\end{remark}

The next result yields the Bargmann spectral functions of Toeplitz operators whose symbols are imaginary part independent operators. 

\begin{theorem}\label{thm:ImPartInd-spectral}
	Let us fix $\lambda > -1$, and consider the auxiliary function
	\begin{align}
		\alpha_\lambda : \R_+ &\rightarrow \R  \label{eq:alpha_lambda} \\
		\alpha_\lambda(\xi) &= \xi^{\frac{1}{2}(\lambda + 1)} e^{-\xi}. \notag
	\end{align}
	If $A$ is an imaginary part independent operator of $L^2(\C_+, v_\lambda)$, then the Bargmann spectral function $h_A \in L^\infty(\R_+)$ of the Toeplitz operator $T^{(\lambda)}_A$ is given~by
	\begin{equation}\label{eq:ImPartInd-spectral}
		h_A(\xi) 
		= \frac{(2\sqrt{\xi})^{\lambda + 1}}{\Gamma(\lambda + 1)} e^\xi 
				\int_{\R_+}\big( A(\alpha_\lambda e_*)\big)(x + i \xi) 
					e^{-x\xi} x^\lambda \dif x,
	\end{equation}
	for almost every $\xi \in \R_+$, where we let $e_* : \C_+ \rightarrow \R$ be the function given by $e_*(x + i\xi) = e^{-x\xi}\chi_{\R_+}(\xi)$.
\end{theorem}
\begin{proof} 
	By Proposition~\ref{prop:ImPartInd}(1) we can apply Theorem~\ref{thm:spectral-FourierIntertwining} and conclude that its Bargmann spectral function $h_A \in L^\infty(\R_+)$ is given by the identity \eqref{eq:spectral-FourierIntegral}. To compute $h_A$ from the latter, it is enough to find $A(\Phi)$ for suitable $\varphi \in L^2(\R_+)$, where $\Phi$ is defined in terms of $\varphi$ as in Theorem~\ref{thm:spectral-FourierIntertwining}. 
	
	Let us choose and fix any given $\varphi \in C_c(\R_+)$. Then, the function $b_\lambda : \C_+ \rightarrow \R$ defined by
	\[
		b_\lambda(x + i\xi) = \xi^{\frac{1}{2}(\lambda + 1)} \varphi(\xi) 
				\alpha_\lambda(\xi)^{-1} \chi_{\R_+}(\xi)
			= e^\xi	\varphi(\xi) \chi_{\R_+}(\xi),
	\]
	is bounded and does not depend on the real part of its variable. On the other hand, we compute the following norm in $L^2(\C_+, v_\lambda)$
	\begin{align*}
		\|\alpha_\lambda e_*\|^2 
			&= \frac{\lambda + 1}{4\pi} \int_{\R_+^2}
					\xi^{\lambda + 1} e^{-2\xi} e^{-2x\xi} x^\lambda \dif x \dif \xi \\
			&= \frac{\lambda + 1}{4\pi}\int_{\R_+} 
				\xi^{\lambda + 1} e^{-2\xi}\frac{\Gamma(\lambda + 1)}{(2\xi)^{\lambda + 1}}
				\dif \xi  \\
			&= \frac{\Gamma(\lambda + 2)}{2^{\lambda + 4}\pi} < +\infty,
	\end{align*}
	which implies that $\alpha_\lambda e_* \in L^2(\C_+, v_\lambda)$.
	
	Now observe that
	\[
		\Phi = b_\lambda \cdot (\alpha_\lambda e_*),
	\]
	where we have proved that the first factor belongs to $L^\infty(\C_+)$ and does not depend on the real part of its variable, while the second factor belongs to $L^2(\C_+, v_\lambda)$. Hence, we can apply Proposition~\ref{prop:ImPartInd-withfunction} to conclude that
	\[
		A(\Phi) = b_\lambda A(\alpha_\lambda e_*).
	\]
	Once this expression has been achieved, we apply \eqref{eq:spectral-FourierIntegral} from Theorem~\ref{thm:spectral-FourierIntertwining} to compute, for almost every $\xi \in \R_+$, as follows
	\begin{align*}
		h_A(\xi) \varphi(\xi)
			&= \frac{(2\sqrt{\xi})^{\lambda + 1}}{\Gamma(\lambda + 1)}
				\int_{\R_+} b_\lambda(x + i \xi) 
					\big(A(\alpha_\lambda e_*)\big)(x + i \xi)
						e^{-x\xi} x^\lambda \dif x  \\
			&= \frac{(2\sqrt{\xi})^{\lambda + 1}}{\Gamma(\lambda + 1)} 
					e^\xi \varphi(\xi)
				\int_{\R_+} \big(A(\alpha_\lambda e_*)\big)(x + i \xi) 
						e^{-x\xi} x^\lambda \dif x,
	\end{align*}
	where we have used the expression for $b_\lambda$ which is independent of the integration variable $x$. Since $\varphi \in C_c(\R_+)$ was arbitrarily given, this yields \eqref{eq:ImPartInd-spectral} and thus completes the proof.
\end{proof}

\begin{remark}\label{rmk:ImPartInd-spectral}
	We observe that the function $e_*$ defined in Theorem~\ref{thm:ImPartInd-spectral} does not belong to $L^2(\C_+, v_\lambda)$ (as a straightforward computation will show) so one cannot apply Proposition~\ref{prop:ImPartInd-withfunction} with $\psi = e_*$ to move the function $\varphi$ outside of the operator $A$, as needed to obtain an explicit formula for $h_A$. This is the reason to introduce the auxiliary function $\alpha_\lambda$. Note that it is possible to consider other choices of such $\alpha_\lambda$. From the proof above, it seems natural to use the factor $\xi^{\frac{1}{2}(\lambda + 1)}$ in \eqref{eq:alpha_lambda}. However, the proof also shows that $\xi \mapsto e^{-\xi}$ may be replaced with any other function belonging to $L^2(\R_+)$ in the definition \eqref{eq:alpha_lambda} of $\alpha_\lambda$.
	
	On the other hand, for a multiplier operator $M_a$, given by $a \in L^\infty(\R_+)$ such that $a(x + i\xi) = \widehat{a}(x)$ (as with our previous notation), we have
	\[
		\big(M_a(\alpha_\lambda e_*)\big)(x + i\xi) 
			= \widehat{a}(x) \xi^{\frac{1}{2}(\lambda + 1)} e^{-\xi} e^{-x\xi}
	\]
	for almost every $x + i\xi \in \C_+$. Substitution of the latter in \eqref{eq:ImPartInd-spectral} leads to the cancellation of the factor function $\xi \mapsto e^{-\xi}$, or of any other suitable choice, and reduces the formula \eqref{eq:ImPartInd-spectral} for $h_{M_a} = h_a$ to the one stated in Corollary~\ref{cor:spectral-FourierIntertwining}, i.e.~the previously known spectral function for Toeplitz operators with translation invariant function symbol.
	
	It seems that Toeplitz operators with imaginary part independent operator symbol introduce the need of an auxiliary function $\alpha_\lambda$ as in Theorem~\ref{thm:ImPartInd-spectral}, at least with our current methods. It would be interesting to find similar formulas that do not depend on such auxiliary functions.
\end{remark}

\subsection*{Acknowledgement}
This research was partially supported by SNII-Secihti and Secihti Grants 61517 and CBF-2025-I-828.

\end{document}